\documentclass[12pt,a4paper,reqno]{amsart}
\usepackage{amsfonts,amsthm,latexsym,amsmath,amssymb,amscd,amsmath}

\usepackage[T2A]{fontenc}
\usepackage[cp1251]{inputenc}
\usepackage{latexsym,amssymb,amsthm}
\usepackage{url,graphics} 
\usepackage[dvips]{graphicx,epsfig}

\newtheorem{thm}{Theorem}[section]
\newtheorem{prop}[thm]{Proposition}
\newtheorem{lem}{Lemma}[section]
\newtheorem{defin}{Definition}[section]

\newtheorem{cor}{Corollary}[section]

\newtheorem{remark}{Remark}[section]
\newtheorem{question}{Question}
\newtheorem{conjecture}[question]{Conjecture}

\newtheorem{Not}[defin]{Notation}

\def\q#1.{{\bf #1.}}
\renewcommand\geq{\geqslant}
\renewcommand\leq{\leqslant}

\newcommand{\bR}{\mathbb{R}}
\newcommand{\bC}{\mathbb{C}}
\newcommand{\bK}{\mathbb{K}}

\newcommand {\HS} {\mathcal {H}}

\newcommand{\cal}{\mathcal}

\newcommand{\lY}{\overline{Y}}

\begin{document}
          \numberwithin{equation}{section}

          \title[Postnikov-Shapiro Algebras  and matroids]
          {Postnikov-Shapiro Algebras, Graphical Matroids and their generalizations}

\author[G. Nenashev]{Gleb Nenashev}
\address{  Department of Mathematics, Stockholm University, SE-106 91 Stockholm, Sweden,}
\email{nenashev@math.su.se}

\begin{abstract} 
A.Postnikov and B.Shapiro introduced a class of commutative algebras which enumerate forests and trees of graphs. Our main result is that the algebra counting forests depends only on graphical matroid and the converse. 

By these algebras we motivate generalization of the definition of spanning  forests and trees for  hypergraphs and the corresponding "hypergraphical" matroid. We present $3$ different equivalent definitions, which can be read independently from other parts of the paper.
\end{abstract}
\maketitle

\setcounter{tocdepth}{1}
\tableofcontents

\section{\bf Introduction}

There are a lot of results  about enumeration of spanning trees  and forests of graphs. 
The famous matrix-tree theorem of Kirchhoff (see~\cite{Kir} and p.~138 in~\cite{Tut}) claims that the number of spanning trees of a given graph $G$  equals  the determinant of the Laplacian matrix of $G$. Many  generalizations of the classical matrix-tree theorem were constructed in a long period,  e.g. for directed graphs,  matrix-forest theorems, etc (see~e.g.~\cite{ChK}, see~\cite{Bur} and references therein).
 It is also well known that the number of maximal spanning forests of $G$ (or equivalently trees for connected $G$)  equals  $T_G(1,1)$ while the number of 
all spanning forests of $G$ equals  $T_G(2,1)$, where $T_G$ is the Tutte polynomial of $G$ (see p.~237 in~\cite{Tut}).

We focus here on algebras introduced by A.\,Postnikov and B.\,Shapiro in ~\cite{PSh} 
and on graphical matroids, both objects enumerate spanning trees and forests. Before our paper there was no direct relation between them, only across the Tutte polynomial, Critical group (Sandpile model). Also with the motivation given by these algebras we introduce $3$ equivalent definitions of spanning trees (forests) of a hypergraph and present corresponding "hypergraphical" matroid. If you are interested only in the latter, scroll to section~\ref{sec:hyper}.

\medskip

The field $\bK$ of zero characteristic is fixed throughout this paper, for example $\bC$ or $\bR$. By a graph we always mean an undirected graph without loops (multiple edges are allowed). We use the standard notation: $E(G)$ and $e(G)$ are the set and the number of edges of graph $G$; $V(G)$ and $v(G)$ are the set and the number of vertices of~$G$; $c(G)$ is the number of connected components of $G$ and $T_G(x,y)$ is the Tutte polynomial of $G$.

At first, we present two definitions of Postnikov-Shapiro algebras counting spanning forests.
\begin{Not}
\label{forest}
 Take an undirected graph $G$ on $n$ vertices.

\emph{(I)} Let $\Phi_{G}^{F}$ be the commutative algebra over $\bK$ generated by $\{ \phi_e: \ e\in E(G) \}$ satisfying the relations $\phi_e^{2}=0$, for any $e\in E(G)$.

Fix any linear order of vertices of $G$. For $i=1,\ldots,n$, set 
$$X_i=\sum_{e\in G} c_{i,e} \phi_e,$$
 where $$
c_{i,e}=\begin{cases} \;\;\;1\quad \text{if}\; e=(i,j), i<j;\\
                                            -1\quad\text{if}\; e=(i,j), i>j;\\
                                             \;\;\;0\quad \text{otherwise}.
\end{cases}
$$
Denote by ${\cal C}_{G}^{F}$ the subalgebra of $\Phi_{G}^{F}$ generated by $X_1,\ldots ,X_n$.

\emph{(II)}   Consider the ideal $J_{G}^{F}$ in the ring $\bK[x_1,\cdots,x_n]$ generated by
$$p_I^{F}=\left(\sum_{i\in I} x_i\right)^{ D_I+1},$$
where $I$ ranges over all nonempty subsets of vertices, and $D_I$ is the total number of edges between vertices in $I$ and vertices outside~$I$.  Define the algebra   ${\cal B}_{G}^{F}$ as the quotient $\bK[x_1,\dots,x_n]/ J_{G}^{F}.$
\end{Not}

\begin{remark}
\label{signs}
In the first definition, the order of vertices is not important. Easy to see that all such algebras are isomorphic to the above definition of ${\cal C}_{G}^{F}$, even though we can separately  choose the "smallest" vertex for each edge. 
For an orientation $\overline{G}$ of graph $G$, we denote  by~${\cal C}_{\overline{G}}^{F}$ the subalgebra of $\{ \phi_e: \ e\in E(G) \}$, where $c_{i,e}=1$ for the end of $e$ and $-1$ for the beginning of $e$. 
\end{remark}

The original motivation of  these algebras is that, for the complete graph $K_n$, the algebra is isomorphic to the cohomology ring of the flag manifold $Fl_n$, see~\cite{PSS} and~\cite{ShSh}. Algebras of complete graphs are also related to Fomin-Kirillov and Orlik-Terao algebras, see their definition in~\cite{FK},~\cite{OT} and relations between the last two in~\cite{RL}. In paper~\cite{Berg} a big class of algebras was considered, which includes our case, Orlik-Terao and others.
Postnikov-Shapiro algebras have direct connections with $G$-parking functions (see~\cite{PSh},~\cite{ChP}), whose are stable configurations of the sandpile model introduced in~\cite{Dhar}. 

There are a few generalizations of $\cal{B}^F_G$ and $\cal{B}^T_G$ in the literature (the latter algebra counts trees in $G$ instead of forests, see notation in section~5), see the most important generalizations  in  papers~\cite{ArP} and~\cite{HR}. In both papers the main object of the generalizations is the algebra of the second definition (quotient algebra), they defined the so called Zonotopal algebras, where dimensions of algebras are numbers of lattice points of corresponding zonotops. 
In~\cite{Hua} the definitions of $\cal{B}^\Delta_G$ and $\cal{C}^\Delta_G$ were given  in terms of some simplex complex $\Delta$ on the set of vertices of $G$. With this definition algebras ${\cal B}_G^F$ and ${\cal B}_G^T$ become particular cases corresponding to different simplexes.

\smallskip

 We need the following classic notation, which was constructed by Tutte for his original definition of the polynomial.

\begin{Not} 
Fix some  linear order on the set $E(G)$ of edges of $G$. Let $F$ be any spanning forest in  $G$.
By $act_G(F)$ denote the number of all externally active edges of $F$, i.e. the number of edges $e\in E(G)\setminus F$ such that \emph{(i)}~subgraph $F+e$ has a cycle; \emph{(ii)} $e$ is the minimal edge in this cycle in the above linear order.

 Denote by $F^+$ the set of edges of the forest $F$ together with all externally active edges,
and denote by $F^-=E(G)\setminus F^+$ the set of externally nonactive edges.
\end{Not}

The following theorem is the main result of~\cite{PSh}, which shows that both definitions are equivalent and why it is called algebra counting spanning forests.

\begin{thm}[cf.~\cite{PSh}]
\label{thmPSh}

 For any graph $G$, algebras  ${\cal B}_{G}^{F}$ and ${\cal C}_{G}^{F}$ are isomorphic,
 their total dimension over $\bK$ is equal to the number of spanning forests in $G$.

Moreover, the dimension of the $k$-th graded component of these algebras  equals
the number of spanning forests $F$ of $G$ with external activity $e(G)-e(F)-k$.
\end{thm}

In fact, the second part of Theorem~\ref{thmPSh} claims that the above Hilbert polynomial is a specialization of the Tutte polynomial 
of~$G$.   
\begin{cor}
\label{tutte-t=1}
Given a graph $G$, the Hilbert series of the algebra $\mathcal{C}_G^F$ is given by
$$\HS_{\mathcal{C}_G^F}(t)=T_G\left(1+t,\frac{1}{t}\right)\cdot t^{e(G)-v(G)+c(G)},$$ 
where $c(G)$ is the number of the connected components of  $G$.
\end{cor}
Both definitions of algebras are important, and sometimes it is easier to use the first and sometimes the second. For example: if you want to use computer algebra, then the first is better, because it has less relations, but in fact, the first definition is subalgebra of the bigger algebra $\Phi_G^F$.

\bigskip
\noindent The structure of the paper is the following:

In \S\,$2$ we show that the Postnikov-Shapiro algebras counting spanning forests and graphical matroids are in one to one correspondence, see Theorem~\ref{matroid}. It means that these algebras store almost all information about the graphs. 

\smallskip
In \S\,$3$ we consider $t$-labelled analouge of algebras, namely we use the condition  $\phi_e^{t+1}=0$ instead $\phi_e^2=0$. The similar properties hold and a new property that we can reconstruct the Tutte polynomial from the Hilbert series of the algebra corresponding to large $t$, see Theorem~\ref{calc}.

\smallskip
In \S\,$4$ we construct a family of algebras corresponding to a given hypergraph, and we 
 present a natural definition of a matroid of a hypergraph such the Hilbert series of algebras is a specialization of the Tutte polynomial corresponding to this matroid. 

\smallskip
In \S\,$5$ we discuss  similar problems for algebras counting spanning trees and formulate Conjecture~\ref{conj}.

\smallskip
{\bf Acknowledgement.} I am grateful to my supervisor Boris~Shapiro for introducing me to this area, for his comments and editorial help with this text.

\bigskip
\section{\bf Algebras and  Matroids}
Obviously, the original Postnikov-Shapiro algebra  corresponding to a disconnected graph $G$ is the Cartesian product of the algebras corresponding to the connected components of $G$. In particular, it is also true for $2$-connected components (maximal connected subgraphs such that they remain connected, after removal of any vertex). The same fact is also true for matroids. In this section we prove the following result.

\begin{thm}
\label{matroid}
 Algebras ${\cal B}_{G_1}^F$ and ${\cal B}_{G_2}^F$ for graphs $G_1$ and $G_2$ are isomorphic if and only if 
the graphical matroids of these graphs coincide.

{\rm P.S.} The  algebraic isomorphism can be thought of either as graded or as non-graded, the statement holds in both cases.
\end{thm}

In the subsequent paper~\cite{NS} filtered algebras are considered, that distinguish graphs. The proof is based on some ideas from this paper and on another trick, that makes a proof easier. 
Here we should proof both sides and in the proof of Theorem~\ref{matroid} we use the following theorem of H.\,Whitney.
\begin{thm}[Whitney's $2$-isomorphism theorem, see~\cite{Wit},~\cite{Oxl}] Let $G_1$ and $G_2$ be two graphs. Then their graphical matroids are isomorphic if and only if $G_1$ can be transformed to a graph, which is isomorphic to $G_2$ by a sequence of operations of vertex identification, cleaving and twisting.
\end{thm}

These three operations are defined below.

\noindent 1a) {\it Identification}: Let $v$ and $v'$ be vertices from different connected components of the graph. We modify the graph by identifying $v$ and $v'$ as a new vertex $v''$. 

\noindent 1b) {\it Cleaving} (the inverse of identification): A graph can only be cleft at a cut-vertex.

\noindent 2) {\it Twisting}: Suppose that the graph $G$ is obtained from two disjoint graphs $G_1$ and $G_2$ by identifying vertices $u_1$ of $G_1$ and $u_2$ of $G_2$ as the vertex $u$ of $G$ and additionally identifying vertices $v_1$ of $G_1$ and $v_2$ of $G_2$ as the vertex $v$ of $G$. In a twisting of $G$ about $\{u, v\}$, we identify $u_1$ with $v_2$ and $u_2$ with $v_1$ to get a new graph $G'$.

\bigskip
We split our proof of Theorem~\ref{matroid} in two parts presented in \S\,$2.1$ and in \S\,$2.2$.

\subsection{\bf Algebras are isomorphic if their matroids are isomorphic.}

Because algebras ${\cal B}^{F}_G$ and ${\cal C}^{F}_G$ are isomorphic for any graph by Theorem~\ref{thmPSh}, it suffices to  prove Theorem~\ref{matroid} for algebras ${\cal C}^{F}_{G_1}$ and ${\cal C}^{F}_{G_2}$.

\begin{lem} 
\label{forestlem}
If graphs $G$ and $G'$ differ by a sequence of Whitney's deformations, 
then the algebras ${\cal C}^{F}_{G}$ and ${\cal C}^{F}_{G'}$ are isomorphic.
\begin{proof}
It is sufficient to check the claim for each deformation separately.

{\it $1^\circ$ Identification and Cleaving.} We need to prove our fact only for cleaving, because identification is the inverse of cleaving.

In this case algebras doesn't change, because the linear subspace defined by $X_i$ for vertices doesn't change. This holds, because if we split a vertex $k$ into $k'$ and $k''$, then in the new graph, $X_{k'}$ equals to the minus sum of $X_i$ corresponding to the vertices from its component except $k'$ (sum of all $X_{.}$ from one connected component is zero), i.e.
$X_{k'}$ belongs to  the linear space ${<}X_1,\ldots,X_{k-1},X_{k+1},\ldots,X_n{>}$. Similarly $X_{k''}$ belongs to the linear space ${<}X_1,\ldots,X_{k-1},X_{k+1},\ldots,X_n{>}$. Hence, ${<}X_1,\ldots,X_{k'},X_{k''},\ldots,X_n{>}$ is a subspace of the linear space ${<}X_1,\ldots,X_n{>}$. The equation, $X_k=X_{k'}+X_{k''}$ implies that these linear spaces coincide.

\smallskip
{\it $2^\circ$ Whitney's deformation of the second kind.}
 Define the digraph $\overline{G}$ as the orientation of the graph $G$, where each arrow goes to the "smallest" vertex (see Remark~\ref{signs}).
 
 Let us make a twist of the vertices $u$ and $v$. Let $\overline{G}_1$ and $\overline{G}_2$ be the orientations of $G_1$ and $G_2$ corresponding to $\overline{G}$.
  Let $\overline{G'}$ be the orientation of $G'$ corresponding to the gluing $\overline{G}_1$ and $\overline{G}_2$ with reversing each arrow from $\overline{G}_2$.
   Vertex $u'$ in $G'$ is obtained by gluing of $u_1$ and $v_2$; $v'$ is obtained by gluing of $v_1$ and $u_2$.
 
 Let $X_k$, $X_{1,k}$, $X_{2,k}$ and $X'_k$ be the sums  with signs of edges incident to vertex $k$ in graphs $\overline{G}$, $\overline{G}_1$, $\overline{G}_2$ and $\overline{G'}$.
 
  \noindent For a vertex $k$ of $G_1$ except $u_1$ and $v_1$, we get
  
   $X_k=X_{1,k}=X'_k$.
 
 \noindent For a vertex $k$ of $G_2$ except $u_2$ and $v_2$, we get
 
 $X_k=X_{2,k}=-X'_k$, because we reverse the orientation in the second part of twisting. 
 
 \noindent For other vertices we have:
 
  $X_u=X_{1,u_1}+X_{2,u_2};$
 
 $X_v=X_{1,v_1}+X_{2,v_2};$
 
 $X'_{u'}=X_{1,u_1}-X_{2,v_2}=X_u-(X_{2,u_2}+X_{2,v_2});$
 
 $X'_{v'}=X_{1,v_1}-X_{2,u_2}=X_v-(X_{2,u_2}+X_{2,v_2}).$
 
\noindent We know that the sum of variables corresponding to the vertices of any graph is zero, because each edge goes with a plus to one vertex and with a minus to another.  We have 
$$\sum_{k\in G_2}X_{2,k}=0,$$
$$\sum_{k\in G_2\setminus {u_2,v_2}}X_{2,k}+X_{2,u_2}+X_{2,v_2}=0,$$
$$X_{2,u_2}+X_{2,v_2}=-\sum_{k\in G_2\setminus \{u_2,v_2\}}X_{2,k}=-\sum_{k\in G_2\setminus \{u_2,v_2\}}X_{k}.$$

Hence, $X'_{u'}$ and $X'_{v'}$ belong to the linear space generated by $X_k$, where $k\in G$. In other words, the linear space for $\overline{G'}$ is a linear subspace of the space for $\overline{G}$. Similarly we can prove the converse. Then, the linear spaces coincide, and since we have the same relations ($\phi_e^2=0$ for any edge),
 the algebras in these bases coincide.
\end{proof}
\end{lem}

We have proved the first part of Theorem~\ref{matroid}, because if the corresponding graphical matroids are isomorphic, then there exists such a sequence of Whitney's operations.

\begin{cor}
The algebra corresponding to a graph $G$ is the Cartesian product of the algebras corresponding to the $2$-connected components of~$G$.
\end{cor}

\smallskip
\subsection{\bf Reconstruction of the matroid}

\begin{lem}
It is possible to reconstruct a matroid of a graph $G$ from the algebra 
${\cal C}_G^F$.

\begin{remark} We assume that we only know ${\cal C}_G^F$ as an algebra. I.e.  we assume that we do not know the basis corresponding to the vertices of $G$, and that we have no information about the graded components of~${\cal C}_G^F$.
\end{remark}
\begin{proof}
For an element $Y\in {\cal C}_G^F$, we define its length $\ell(Y)$ as the minimal number such that $Y^{\ell+1}$ is zero (the length can be infinite).

We call an element $Y\in {\cal C}_G^F$ irreducible if there is no representation $Y=\sum_i Z_{2i-1}Z_{2i}$ such that $\ell(Z_j)$ is finite for any $j$.

Consider a basis $\{Y_1,...,Y_m\}$ of the algebra ${\cal C}_G^F$ with the following conditions:
\begin{itemize}

\item Each $Y_j$ is irreducible; 

\item For any $k\leq m$, different nonzero numbers $ r_1,\ldots,r_k \in \bK$, any $ r'_1,\ldots,r'_k \in \bK$ and different $i_1,\ldots,i_k\in [m]$, we have$$\ell(r_1Y_{i_1}+\cdots+r_kY_{i_k})\geq \ell(r'_1Y_{i_1}+\cdots+r'_kY_{i_k});$$
\item For any  linear combination $Y$ of $\{Y_1,...,Y_m\}$ and a reducible~$Z$,
$$\ell(Y)\leq \ell(Y+Z);$$
\item  $\sum_{i} \ell(Y_i)$ is minimal.
\end{itemize}

Such a basis of ${\cal C}_G^F$ always exists. For example, the basis $X_1,\ldots, X_n$ (corresponding to the vertices) satisfies the first three conditions. However, the sum of lengths of $X_i$ is not minimal.

For an element $Y\in {\cal C}_G^F$, we define its linear part as $\lY$ and its remainder as $\widehat {Y}$, where $Y=\lY+\widehat {Y}$, $\overline {Y}$ belongs to the $1$-graded component of ${\cal C}_G^F$  and  $\widehat {Y}$ belongs to the linear span of the other graded components. 
Observe that we do not know this decomposition explicitly, because we do not know the $1$-graded component.  
We say that an edge $e$ belongs to $\lY$ if $\lY$  includes the variable $\phi_e$ corresponding to the edge $e$ with a nonzero coefficient.
\begin{prop} 
\label{basis} A basis $\{Y_1,...,Y_m\}$ as above satisfies additionally the following conditions:
\begin{enumerate}
\item The set $\{\lY_1,...,\lY_m\}$ is a basis;

\item For any linear combination $Y$ of $\{Y_1,...,Y_m\},$ 
$$\ell(Y)= \ell(\overline{Y});$$

\item Each edge belongs to one or two $\lY_i$ and $\lY_j$. If it belongs to two, then with the opposite coefficients; 

\item Each $\lY_i$ has edges only from one $2$-connected component.

\end{enumerate}
\begin{proof}
{\bf (1).}
We know that $\{Y_1,...,Y_m\}$ is a basis. 
Hence, the $1$-graded component of ${\cal C}_G^F$ coincides with a linear span of ${<}\lY_1,...,\lY_m {>}$. Then, ${<}\lY_1,...,\lY_m {>}$ is also a basis.

{\bf (2).} For any $i\in [m]$, the part $\widehat{Y_i}$ is reducible, otherwise $\ell (Y_i)$ is infinite, and the sum in the last condition is infinite. However for the basis corresponding to the vertices this sum is finite. Then for any linear combination $Y$ of $\{Y_1,...,Y_m\}$, we have $\ell(Y)\leq \ell(\overline{Y})<\infty$ (by third condition of choice). However, it is clear that $\ell(Y)\geq \ell(\overline{Y})$. Hence, $\ell(Y) = \ell(\overline{Y})$ for any $Y$ from the linear space  ${<}Y_1,...,Y_m{>}$.

\smallskip
{\bf (3).} Obviously, each edge belongs to at least one $\lY_{j}$, because any edge belongs to some $X_i$ and this $X_i$ is a linear combination  of $\lY_1,...,\lY_m$.

Assume that there is an edge $e$, which belongs to $\lY_{i_1}$, $\lY_{i_2}$ and $\lY_{i_3}$. Then there are different nonzero numbers $r_1$, $r_2$ and $r_3$ such that, for $Y=r_1Y_1+r_2Y_2+r_3Y_3$, $\overline{Y}$ without $e$. Then $\ell(\overline{Y})$ is at most the number of all edges in  $\lY_{i_1}$, $\lY_{i_2}$ and $\lY_{i_3}$ minus one.

Consider $r'_1$, $r'_2$ and $r'_3$ in general position, then, for $Y'=r'_1Y_1+r'_2Y_2+r'_3Y_3$, $\ell(\overline{Y'})$ is the number of all edges in  $\lY_{i_1}$, $\lY_{i_2}$ and $\lY_{i_3}$. We have
$$\ell(\overline{r_1Y_1+r_2Y_2+r_3Y_3})<\ell(\overline{r'_1Y_1+r'_2Y_2+r'_3Y_3}).$$
Using {(2)} we also have
$$\ell({r_1Y_1+r_2Y_2+r_3Y_3})<\ell({r'_1Y_1+r'_2Y_2+r'_3Y_3}),$$
which contradicts our choice of the basis.

\smallskip
We proved the first part of condition~(3); the proof of the second part is the same, but only for two different $r_1$ and $r_2$.

\smallskip
{\bf (4).}
Assume the opposite, i.e. that there is $\lY_{i}$ which has edges belonging to two different $2$-connected components.

We know that the algebra is the Cartesian product of the subalgebras corresponding to $2$-connected components. Then there are $Z_1$ and $Z_2$ from the $1$-graded component, such that $\lY_{i}=Z_1+Z_2$ and $\ell(Z_1),\ell(Z_2)<\ell(\lY_{i})$.(For example, $Z_1$ is a part corresponding to some $2$-connected component and $Z_2$ is another part). 

Because $Z_1$ and $Z_2$ belong to the linear space ${<}\lY_1,...,\lY_m{>}$, we can change  $\lY_{i}$ to $Z_1$ or $Z_2$ in the basis $\{\lY_1,...,\lY_m\}$. (Indeed if we can not do it, then $Z_1$ and $Z_2$ belong to ${<}\lY_{1},\ldots \lY_{{i-1}},\lY_{{i-1}}, \ldots\lY_{m}{>}$. 
Therefore $\lY_{{i}}=Z_1+Z_2$ also belongs to the latter space). Let us $\{\lY_{1},\ldots \lY_{{i-1}}, Z_1, \lY_{{i-1}}, \ldots\lY_{m}\}$ is also a basis.

We have a new basis $\{\lY_{1},\ldots \lY_{{i-1}}, Z_1, \lY_{{i-1}}, \ldots\lY_{m}\}$, whose sum of lengths is less than the sum of lengths of $\{\lY_{1},\ldots ,\lY_{m}\}$, which is equal to the sum of lengths of $\{{Y_1},\ldots ,{Y_m}\}$. And for this basis, the first three conditions of a choice of basis  hold, and the sum of lengths is smaller. Then we should choose the basis $\{\lY_{1},\ldots \lY_{{i-1}}, Z_1, \lY_{{i-1}}, \ldots\lY_{m}\}$ instead of $\{{Y_1},\ldots ,{Y_m}\}$. Contradiction. 
 \end{proof}
\end{prop}

Let us now construct the cut space of $G$. This will finish the proof, because we can define the graphical matroid in terms of the cut space of a graph. By a cut we mean a set $C$ of edges such that the subgraph $G\setminus C$ has more connected components than $G$. By an elementary cut we mean a minimal cut, i.e. a cut, whose arbitrary subset is not a cut.

The sum $\overline{\sum 2^i Y_i}=\sum 2^i\overline{ Y_i}$ has each edge with a nonzero coefficient by~{(2)} of Proposition~\ref{basis}. Hence, 
$$e(G)=\ell  \left( \sum_{i=1}^m 2^i \lY_{i} \right) =\ell \left(\sum_{i=1}^m 2^i Y_i\right).$$
Therefore we know the number of edges in the graph.

Consider the set $\tau=\{\psi_1,\ldots,\psi_{e(G)}\} $ consisting of $e(G)$ elements and a family of subsets $K_1,\ldots,K_m$
constructed by the following rules.
\begin{itemize}
\item For each pair $i$ and $j$, we prescribe $\frac{\ell(Y_i)+ \ell(Y_j)-\ell(Y_i+Y_j)}{2}$ own elements from $\tau$ and add them to both $Z_i$ and $Z_j$; 

\item For every $i$, we choose $\ell(Y_i)-\sum_{j\neq i} (\frac{\ell(Y_i)+ \ell(Y_j)-\ell(Y_i+Y_j)}{2})$ own elements from $\tau$ and add them to $Z_i$.

\end{itemize}

In fact, for any edge  $e$ from $G$, we choose the corresponding element $\psi_{k(e)}$ and add it to $Z_i$ if and only if $e$ belongs to $\lY_{i}$. 

Consider the space $\Gamma$ of subsets in $\tau$ with the operation $\Delta$ (symmetric difference) generated by $Z_1,\ldots,Z_m$. We want to prove that $\Gamma$ is isomorphic to the cut space of $G$.

\smallskip
Let $C$ be an elementary cut of $G$. 
Let  $X_C$ be the sum of $X_i$ corresponding to the vertices, which belong to some new component of $G\setminus C$ (hence, $X_C$ is in $1$-graded component).
Then $X_C$ has an edge with a nonzero coefficient if and only if the edge belongs to~$C$. Consider the minimal $t$ such that there is a linear combination $X_C=a_1\lY_{{i_1}}+\ldots+a_t\lY_{{i_t}}$; consider the sum $X'_C=\lY_{{i_1}}+\ldots+\lY_{{i_t}}$. Obviously, $X'_C$  is nonzero and has edges only from the cut $C$, because an edge belongs to the sum $X'_C$ if and only if it is exactly in one of $\lY_{{i_1}},\ldots,\lY_{{i_t}}$. 

Assume that $X'_C$ has not all edges from $X_C$. Let $C'$ be a subset of the edges of $C$ belonging to $X'_C$. The set of edges $C'$ is not a cut of $G$, because $C$ is an elementary cut. Hence, for any edge $e\in C'$, there is a path $e_1,\ldots,e_k$ in $G\setminus C'$, such that $e,e_1,\ldots,e_k$ is a simple cycle.  Let the edge $e$ connect vertices $b_k$ and $b_0$, and the edge $e_i$ connect $b_{i-1}$ and $b_i$ for any $i\in[k]$. 
Because $X'_C$ belongs to the $1$-graded component,  $X'_C=\sum_{i=1}^n\widehat{a_i}X_i$, where $\widehat{a_i}\in\bK$ and $X_i$ are the elements corresponding to the vertices of $G$. 
For $i\in [k]$, the edge $e_i$ does not belong to $X'_C$, however variable $\phi_{e_i}$ belongs only to $X_{b_{i-1}}$ and $X_{b_{i}}$ from $\{X_1,\ldots,X_n\}$. 
Furthermore it belongs to one of $\{X_{b_{i-1}},X_{b_{i}}\}$ with coefficient $1$ and to another with $-1$, hence, $\widehat{a}_{b_{i-1}}=\widehat{a}_{b_{i}}$. Then, we also have $\widehat{a}_{b_0}=\widehat{a}_{b_k}$. 
Hence the variable $\phi_e$ belongs to $X'_C$ with a zero coefficient. Contradiction.

\smallskip
We concluded that a subset of $\tau$ corresponding to an elementary cut belongs to the space $\Gamma$.  To finish the proof, we need to show, that if a subset of $\tau$ belongs to $\Gamma$, then it either corresponds to a cut or to the empty set.

Assume the contrary, i.e. assume that there is $Z_{i_1}\Delta Z_{i_2}\Delta\cdots Z_{i_s}$ is not a cut. Let $C$ be a set of edges from 
$\lY_{{i_1}}+\lY_{{i_2}}+\cdots +\lY_{{i_s}}$, then $C$ is not a cut in $G$. By Proposition~\ref{basis} we can split the summation into the summations inside individual connected components (and even inside $2$-connected component). Then, for any connected component it is not a cut.

Let $\lY_{{i_{j_1}}}+\lY_{{i_{j_2}}}+\cdots +\lY_{{i_{j_r}}}$ corresponds to a connected component $G'$, and has the set of edges $C'$. 
Then $$\lY_{{i_{j_1}}}+\lY_{{i_{j_2}}}+\cdots +\lY_{{i_{j_r}}}=\sum_{v\in V(G')}a_vX_v.$$
  We know that $G'\setminus C'$ is connected. 
  Therefore there is a spanning tree $T$ in $G'\setminus C'$. 
  For any edge ${v_i}{v_j}$ from $T$, we have $a_{v_i}=a_{v_j}$; otherwise the edge ${v_i}{v_j}$ belongs to $\sum_{v\in V(G')}a_vX_v$ with a nonzero coefficient. 
  Since $T$ is a spanning tree of $G'$,  all coefficients $a_v$ are the same.
   Thus $\sum_{v\in V(G')}a_vX_v=a(\sum_{v\in V(G')}X_v)=0$; the last sum is zero, because the sum of variables corresponding to vertices from a connected component is zero. 
   Then  we also have $\lY_{{i_1}}+\lY_{{i_2}}+\cdots +\lY_{{i_s}}=0$, hence, $Z_{i_1}\Delta Z_{i_2}\Delta\cdots Z_{i_s}$ is the empty set. 
  
  Therefore the space $\Gamma$ is isomorphic to the cut space of $G$, i.e. there is a unique graphical matroid corresponding to ${\cal C}^F_G$. 
  \end{proof}
\end{lem}

\bigskip
\section{\bf Algebras associated with $t$-labelled forests}

In this section we substitute the square-free algebra $\Phi_G^F$ for the $(t+1)$-free algebra $\Phi_G^{F_t}$.  \begin{Not}
\label{tforest}
 Take an undirected graph $G$ on $n$ vertices. Let $t>0$ be a positive integer.

\emph{(I)} Let $\Phi_{G}^{F_t}$ be a commutative algebra over $\bK$ generated by $\{ \phi_e: \ e\in E(G) \}$ satisfying the relations $\phi_e^{t+1}=0$, for any $e\in E(G)$.

Fix any linear order on the vertices of $G$. For $i=1,\ldots,n$, set 
$$X_i=\sum_{e\in G} c_{i,e} \phi_e,$$
 where $c_{i,e}$ as in the notation~\ref{forest}.
Denote by ${\cal C}_{G}^{F_t}$ the subalgebra of $\Phi_{G}^{F_t}$ generated by $X_1,\ldots ,X_n$.

\emph{(II)}   Consider the ideal $J_{G}^{F_t}$ in the ring $\bK[x_1,\cdots,x_n]$ generated by
$$p_I^{F_t}=\left(\sum_{i\in I} x_i\right)^{ tD_I+1},$$
where $I$ ranges over all nonempty subsets of vertices, and $D_I$ is the total number of edges between vertices in $I$ and vertices outside~$I$.  Define the algebra   ${\cal B}_{G}^{F_t}$ as the quotient $\bK[x_1,\dots,x_n]/ J_{G}^{F_t}.$

\end{Not}
It enumerates the so-called $t$-labelled spanning forests.


Consider a finite labelling set $\{1,2,\ldots,t\}$ containing $t$ different labels.

\begin{defin}
A spanning forest of a graph $G$ with a label from $\{1,2,\ldots,t\}$  on each edge
is called a $t$-labelled forest.
The weight of a $t$-labelled forest $F$, denoted by $\omega(F)$, is the sum of the labels of all its edges. 
\end{defin}

\begin{thm} 
For any graph $G$ and a positive integer $t$, algebras  ${\cal B}_{G}^{F_t}$ and ${\cal C}_{G}^{F_t}$ are isomorphic.
 Their total dimension over $\bK$ is equal to the number of $t$-labelled forests in $G$.

\smallskip
The dimension of the $k$-th graded component of the algebra ${\cal B}_{G}^{F_t}$ is equal to
the number of $t$-labelled forests $F$ of $G$ with the weight $t\cdot (e(G)-act_G(F)) - k$.

\label{tcol}
\begin{proof}
Denote by $\widehat{G}$ the graph on $n$ vertices and with $t\cdot e(G)$ edges such that each edge of $G$ corresponds to its $t$ clones 
in the graph $\widehat{G}$. In other words, each edge of $G$ is substituted by its $t$ copies with labels $1,2,\ldots,t$. For each edge $e\in E(G)$, its clones $e_1, \dots, e_t\in E(\widehat{G})$  are ordered according to their numbers; 
clones of different edges have the same linear order as the original edges. Thus we obtain a linear order of the edges of~$\widehat{G}$.

Consider the following bijection between $t$-labelled forests in $G$ and usual forests in $\widehat{G}$:
each $t$-labelled forest $F\subset G$ corresponds to the forest $F'\subset  \widehat{G}$, such that
for each edge $e\in E(F)$, the forest $F'$ has the clone of the edge $e$ whose number is identical to the label of the edge $e$ in the forest $F$.
 
Obviously,
$$act_{\widehat{G}}(F')=t\cdot act_G(F)+\omega(F)-e(F),$$
and $e(\widehat{G})=t\cdot e(G)$. Since ${\cal B}_{G}^{F_t}$ and ${\cal B}_{\widehat{G}}^F$ are the same, the 
Hilbert series of the algebra ${\cal B}_{G}^{F_t}$ coincides with the Hilbert series of the algebra ${\cal B}_{\widehat{G}}^F$, 
which settles the second part of Theorem~\ref{tcol}.

To prove the first part of this Theorem, observe that ${\cal B}_{G}^{F_t}$ and ${\cal B}_{\widehat{G}}^F$ are isomorphic, and algebras 
${\cal C}_{\widehat{G}}^{F}$ and ${\cal B}_{\widehat{G}}^F$ are isomorphic. Thus we must show that algebras  ${\cal C}_{\widehat{G}}^{F}$ and ${\cal C}_{G}^{F_t}$ are isomorphic.
 This is indeed true, because for every edge $e\in E(G)$, the elements $\phi_e,\dots ,\phi_e^{t}$ are linearly independent in the algebra $\Phi_{G}^{F_t}$ with coefficients containing no $\phi_e$. 
Also elements $(\phi_{e_1}+\dots+\phi_{e_t}),\dots ,(\phi_{e_1}+\dots+\phi_{e_t})^{t}$  are linearly independent in the algebra $\Phi_{\widehat{G}}^{F}$ with coefficients containing no $\phi_{e_1},\dots,\phi_{e_t}$, and  $(\phi_{e_1}+\dots+\phi_{e_t})^{t+1}=0$. Moreover  elements $\phi_{e_i}$ only occur in the sum $(\phi_{e_1}+\dots+\phi_{e_t})$ in the algebra $\Phi_{\widehat{G}}^{F}$.
\end{proof}
\end{thm}

\bigskip

In fact Hilbert Series of $\cal{B}_G^{F_t}$ was calculated in papers~\cite{ArP} and~\cite{Lenz}. Furthermore, the Hilbert Series was computed
for $\widehat{G}$, where each edge is replaced by its own number of edges. The  Hilbert series is a specialization of the multivariate Tutte polynomial (see definition in~\cite{S}). 
When "$t$" is the same for every edge, the multivariate Tutte polynomial is calculated from the usual Tutte polynomial. So in our case the Hilbert Series of $\cal{B}_G^{F_t}$ is a specialization of the Tutte polynomial of $G$.
  
\begin{thm} 
\label{tutt}
The dimension of the $k$-th graded component of ${\cal B}_{G}^{F_t}$ is equal to the coefficient of the monomial $y^{t\cdot e(G) -v(G)+c(G) -k}$ 
in the polynomial $$\left( \frac{y^t-1}{y-1} \right)^{v(G)-c(G)} \cdot T_G \left( \frac{y^{t+1}-1}{y^{t+1}-y}, y^t \right).$$

\end{thm}

Consider the graph $\widehat{G}$ constructed in the proof of Theorem~\ref{tcol}. 
We need the following technical lemma which was proved in~\cite{BrOx}.
\begin{lem} [Lemma 6.3.24 in~\cite{BrOx}]
$$T_{\widehat{G}}(x , y)= \left( \frac{y^t-1}{y-1} \right)^{v(G)-c(G)}\cdot  T_G \left( \frac{y^{t}-y+x(y-1)}{y^{t}-1}, y^t \right). $$
\end{lem}

After the substitution $x\to 1+\frac{1}{y}$, we get the following equality. 
\begin{cor} 
\label{=}
$$T_{\widehat{G}}\left(1+\frac{1}{y} , y\right)= \left( \frac{y^t-1}{y-1} \right)^{v(G)-c(G)}\cdot  T_G \left( \frac{y^{t+1}-1}{y^{t+1}-y}, y^t \right). $$ 
\end{cor}
\begin{proof}[Proof of Theorem~\ref{tutt}]
Algebra ${\cal B}_{G}^{F_t}$ is isomorphic to the algebra ${\cal B}_{\widehat{G}}^{F}$ (which was shown in the proof of Theorem~\ref{tcol}), furthermore they are isomorphic as graded algebras. 
So it is enough to show that dimension of the $k$-th graded component of ${\cal B}_{\widehat{G}}^{F}$ is equal to the coefficient of the monomial $y^{t\cdot e(G) -v(G)+c(G) -k}=y^{e(\widehat{G}) -v(\widehat{G})+c(\widehat{G}) -k}$ 
in the polynomial $ T_{\widehat{G}}(1+\frac{1}{y} , y).$ This fact is true by Corollary~\ref{tutte-t=1}. 
\end{proof}

\begin{thm}
\label{calc}
For any positive integer $t\geq n$, it is possible to restore the Tutte polynomial of any connected graph $G$ on $n$ vertices 
knowing only the dimensions of each graded component of the algebra ${\cal B}_{G}^{F_{t}}$.

\begin{proof}
By Theorem~\ref{tcol} we know that the degree  of  the maximal nonempty graded  component of ${\cal B}_{G}^{F_{t}}$
is equal to the maximum of $t\cdot(e(G)-act_G(F))-\omega(F)$ taken over $F$. It attains its maximal value for the empty forest (i.e. $F=\emptyset$). 
Then we know the value of $t\cdot e(G)$, and hence, we know the number of edges of the graph $G$.

Since we know that $t\cdot e(G) -v(G)+c(G)=t\cdot e(G) -n+1$ ($G$ is connected, i.e., $c(G)=1$),
by Theorem~\ref{tutt} we can calculate the polynomial $$\left( \frac{y^t-1}{y-1} \right)^{v(G)-c(G)} \cdot T_G \left( \frac{y^{t+1}-1}{y^{t+1}-y}, y^t \right)$$ from the Hilbert series. Then we  can also calculate $T_G \left( \frac{y^{t+1}-1}{y^{t+1}-y}, y^t \right)$.

It is well known that for any graph $G$, its Tutte polynomial $T_G(x,y)$ is equal to  $\sum_F(x-1)^{c(F)-c(G)}y^{act_G(F)}$, where 
the summation is taken over all spanning forests of $G$. Then we obtain

$$T_G \left( \frac{y^{t+1}-1}{y^{t+1}-y}, y^t \right)=\sum_F \left(\frac{y^{t+1}-1}{y^{t+1}-y}-1\right)^{n-1-e(F)}y^{t\cdot act_G(F)}=$$ $$=\sum_F \left(\frac{1}{y(y^{t-1}+\cdots+1)}\right)^{n-1-e(F)}y^{t\cdot act_G(F)}.$$

Hence, we can restore the polynomial $$\sum_F \left(y^{t-1}+\cdots+1\right)^{e(F)}y^{t\cdot act_G(F)+e(F)}.\eqno (*)$$

Since $e(F)<t$, we can compute the number of usual spanning forests with a fixed pair of parameters $e(F)$ and $act_G(F)$.
Indeed, consider the monomial of the minimal degree in the polynomial~$(*)$, and represent it in the form $s\cdot y^m$. Observe that $s$ is the number of spanning forests $F$ such that $F\equiv m\ (mod\ t)$ and $act_G(F)=\left[\frac{m}{t}\right].$ Remove from the polynomial~$(*)$ all summands for these spanning forests, and repeat this operation until we get~$0$.

Note again that
$T_G(x,y)=\sum_{a,b} \#\{ F : e(F)=a, \ act(F)=b\}\cdot (x-1)^{n-1-a} \cdot y^b$. Therefore we know the whole Tutte polynomial of~$G$, since we know the number of usual spanning forests with any fixed number of edges and any fixed external
activity.
\end{proof}
\end{thm}

\bigskip
\section{\bf Vector configurations and Hypergraphs}
\subsection{\bf Algebra corresponding to vector configuration}

The following algebra was introduced by A.\,Postnikov, B.\,Shapiro and M.\,Shapiro in~\cite{PSS}.

\begin{Not}
Given  a finite set $A=\{a_1,\ldots, a_m\}$ of vectors  in $\bK^n$, let $\Phi_{m}^{F}$ be the  commutative algebra over $\bK$ generated by $\{ \phi_i: \ i\in [m] \}$ with relations $\phi_i^{2}=0$, for each $i\in [m]$.

For $i=1,\ldots,n$, set 
$X_i=\sum_{k\in [m]} a_{k,i} \phi_k.$
Denote by ${\cal C}_{A}$ the subalgebra of $\Phi_{m}^{F}$ generated by $X_1,\ldots ,X_n$.
\end{Not}

The Hilbert series of ${\cal C}_{A}$ also corresponds to a specialization of the Tutte polynomial of the corresponding vector matroid, see Theorem~3 in~\cite{PSS}.

\begin{thm}[cf.~\cite{PSS}] \label{vector}
The dimension of algebra ${\cal C}_A$  is equal to the number of independent subsets in $V$. 
Moreover, the dimension of the $k$-th graded component is equal to the number of independent subsets $S\subset A$ such that $k = m - |S| - act(S)$.
\end{thm}
\begin{cor}
Given a vector configuration $A$ in $\bK^n$, the Hilbert series of algebra ${\cal C}_A$ is
$$\HS_{{\cal C}_A}(t)=T_A\left(1+t,\frac{1}{t}\right)\cdot t^{|A|-rk(A)},$$ 
where $T_A$ is the Tutte polynomial corresponding to $A,$ $|A|$ is the number of vectors in the configuration and $rk(A)$ is the dimension of the linear span of these vectors.
\end{cor}

The set of different vector configurations depends on continuous parameters.
Additionally,  there are uncountably many non-isomorphic algebras each corresponding to its vector configuration. At the same time the number of matroids is countable, and  it is finite for a fixed number of vectors.
 It means that there are many different vector configurations with the same corresponding matroid. 
 In other words, it is in principle impossible  to reconstruct a vector configuration and its algebra from the corresponding matroid.

\medskip
\subsection{\bf Hypergraphs}
\label{sec:hyper}

In this subsection we present a family of algebras corresponding to a hypergraph.
Almost all algebras from this family (generic algebras)  have the same Hilbert series and this generic Hilbert series counts forests of this hypergraph. There are many definitions of spanning trees of a hypergraph,
for example: a spanning cacti in~\cite{Abd}; a hypertree
in~\cite{B&CO} (also known as an arboreal hypergraph
in~\cite{Ber}). However all these definitions allow trees to have different number of edges, whence  spanning tree of a usual graph should have the same number of edges. We define
spanning trees such that this property holds and also other
natural properties hold. Also we define the hypergraphical matroid and the corresponding Tutte polynomial, whose points $T(2,1)$ and $T(1,1)$ calculate the
numbers of spanning forests and of spanning trees, resp. Similar definition of spanning trees and forests was presented in~\cite{Kal}, for that definition there is also Tutte polynomial for a hypergraph, however, there is no matroid.

\smallskip

First we define the family of algebras.

\medskip

Given a hypergraph $H$ on $n$ vertices, let us associate commuting variables $\phi_e, e\in H$ to all edges of $H$.

 Set $\Phi_H$  be the algebra generated by $\{ \phi_e: \ e\in H \}$ with relations $\phi_e^{2}=0$, for any $e\in H$.

Define $C=\{c_{i,e}\in \bK: i\in [n],\  e\in H\}$ as a {\it set of parameters of $H$}, for any edge $e\in H$,  $c_{i,e}=0$ for
 vertices non-incident to $e$, and $\sum_{i=1}^{n} c_{i,e}=0$.
 
For $i=1,\ldots,n$, set
$$X_i=\sum_{e\in H} c_{i,e} \phi_e,$$
 
Denote by ${\cal C}_{H(C)}^F$ the subalgebra of $\Phi_{H}$ generated by 
$X_1,\ldots ,X_n$, and denote by $\widehat{{\cal C}}_{H}^F$ the family of such subalgebras.

The following trivial properties hold for this family of algebras.

\begin{prop}
\label{dim}
\emph{(I)} For a hypergraph $H$,
the dimension of the space of parameters is $\sum_{e\in E} (|e|-1)$.

\emph{(II)} Given a set of parameters $C$ and non-zero numbers $a_e$,~$e\in E$, let $C'$ be the set of parameters such that $c'_{i,e}=a_e c_{i,e}$ for any $i\in [n]$
 and $e\in E$. Then the subalgebras for $C$  and for $C'$  are isomorphic.
\end{prop}

\begin{cor}
For a usual graph $G$,  almost all algebras from $\widehat{{\cal C}}_{G}$ are isomorphic to ${\cal C}_{G}$.
\end{cor}

We define a hypergraphical matroid using the definition of an independent set of edges of a hypergraph.
\begin{defin} Let $H$ be a hypergraph on $n$ vertices. A set $F$ of edges is called {\it independent} if there is a set of parameters $C$ of $H$, such that vectors corresponding to edges from $F$ are linearly independent. In other words, $F$  is  {independent} if, for a generic set of parameters of $H$, vectors are linearly independent. Define the {\it hypergraphical matroid} of $H$ as the matroid with the ground set $E(H)$.
\end{defin}

There is a combinatorial definition of an independent set of edges. First we need to define a cycle of $H$.
 \begin{defin} A subset of edges $C\subset E$ is called a {\it cycle} if 
\begin{itemize} 
\item $|C|=|\cup_{e\in C} e|$

\item There is no subset $|C'\subset C|,$ such that the first property  holds for $C'$. 
\end{itemize}
\end{defin}
Definitions of dependence and of a cycle are similar.
\begin{thm}
\label{Dep=Cyc}
A subset of edges $X\subset E$ is dependent if and only if there is a cycle~$C\subset X$.
\end{thm}
We present a proof of this theorem after  Theorem~\ref{map}.

\begin{defin} A set of edges $F$ is called a {\it spanning forest} if $F$ has no cycles, in other words, $F$ is forest if and only if $F$ is an independent set (by Theorem~\ref{Dep=Cyc}). 
A set of edges $T\subset H$ is called a {\it spanning tree} if it is a forest and $T$ has exactly $v(H)-1$ edges.

A hypergraph $H$ is called {\it strongly connected} if it has at least one spanning tree. 
\end{defin}

\begin{prop}\label{forest-in-tree} 
Maximal spanning forests of a hypergraph have the same number of edges.
In fact, if $H=(V,E)$ is a strongly connected hypergraph, then for any spanning forest $F\subset E$ there is a spanning tree $T$
which contains $F$ (i.e. $F\subset T\subset E$).
\end{prop}
\begin{proof}
By Theorem~\ref{Dep=Cyc} we know that a spanning forest is the same as an independent set of edges. Then we can add edges to a forest until the number of edges
is less than the dimension of the linear space.
  Then all maximal spanning forests have the same size. Hence, if a hypergraph is strongly connected, then any spanning forest is contained in some spanning tree. 
\end{proof}

\smallskip
The Hilbert series of algebras in $\widehat{\cal C}_{H}^F$ are also counting forests of~$H$.
\begin{thm}\label{H:genHS1}
For a hypergraph $H$, generic algebras from $\widehat{{\cal C}}_{H}^F$ have the same Hilbert series.
The dimension  of the $k$-th graded component of a generic algebra equals
the number of spanning forests $F$ in $H$ with the  external activity $e(H)-e(F)-k$.  
\end{thm}
\begin{proof}
By Theorem~\ref{Dep=Cyc} we can change the definition of spanning forests to independent sets.  Consider a generic set of parameters $C$. By Theorem~\ref{vector} we know the Hilbert series of ${\cal C}_{H(C)}^F$ and it is the same for all generic sets of parameters.
\end{proof}

We define the Tutte polynomial of $H$ as the Tutte polynomial of the corresponding hypergraphical matroid.  By the theorems above we know that
\begin{itemize}
\item $T_H(2,1)$ is the number of spanning forests;
\item $T_H(1,1)$ is the number of maximal spanning forests. In fact, $T_H(1,1)$ is the number of spanning trees if $H$ is srongly connected.
\end{itemize}

\smallskip
By Theorem~\ref{H:genHS1}, we get that a generic Hilbert series is a specialization of the Tutte polynomial of $H$.
\begin{cor}
Given a hypergraph $H$ and its generic set of parameters $C$, the Hilbert series of the algebra ${\cal C}_{H(C)}^F$ is given by
$$\HS_{{\cal C}_H(C)^F}(t)=T_G\left(1+t,\frac{1}{t}\right)\cdot t^{e(H)-rk_H},$$ 
where $rk_H$ is the size of a maximal spanning forest of $H$.
\end{cor}

There is another definition of forests/trees of $H$, which again shows that it is a generalization of forests/trees of a usual graph.
\begin{thm}\label{map}
A subset of edges $X\subset E$  is a forest (tree) if and only if there is map from edges to pairs: $e_k\to (i,j),$ where $v_i,v_j\in e_k$, such that these pairs form a forest (tree) in the complete graph  $K_{n}$. 
\end{thm}
\begin{proof}
Consider our forest $F$ and the hypergraph $H$, add to them $n-1-e(F)$ full edges, i.e. edges of type $V$.
We get a new hypergraph $H'$ and a subset of edges $F'$. It is clear that $F'$ is a tree, because there is no cycle without new edges, 
and with at least one new edge we need to cover all vertices and their number is bigger than number of edges. So let $F'=T$ and we will prove our Theorem for a spanning tree.

\smallskip Consider the bipartite graph $B$ with two sets of vertices:
the first set are edges of $T$ and  the second are vertices $V\setminus v_1$, where there is an edge between $e_i$ and $v_j$ if and only if $v_j \in e_i$.
 
There is a perfect matching in $B$, because we can use Hall's marriage theorem (see~\cite{Hall} or any classical book).
 We know that $e(T)=|V\setminus v_1|$ and, for any $X\subset T$, they cover at least $e(X)$ vertices 
(otherwise these edges cover these vertices and may be $v_1$ and then $X$ has cycle). 
Consider a bijection $f$ from $T$ to $V\setminus v_1$ constructed by this perfect matching.
Now we construct by recursion a map $g$ from $T$ to pairs of vertices: 

\begin{enumerate} 
\item $A:=\{v_1\}$ and $ B:=T$ 
\item repeat until $A\neq V:$ 
\begin{itemize} 
\item choose the minimal edge $e_i$ from $B$ such that $e_i\cap A \neq \emptyset$ 
\item $g(e_i)=(u,f(e_i))$, where $u\in e_i\cap A$ 
\item $A:=A\cup \{f(e_i)\}$ and $B:=B\setminus \{e_i\}$ 
\end{itemize}
\end{enumerate} 
It works, otherwise we can not chose such an edge
$e_i$, then either $B=\emptyset$ or $B\neq \emptyset$ at this
moment. We know that $|A|+|B|=n$, then in the first case we already
have $A=V$; in the second case edges from $B$ have vertices only
from $V\setminus A$, then there is a cycle on these edges, i.e. $T$ is not a tree.
Then this algorithm gives some usual tree. 
\end{proof}

\medskip
\begin{proof}[Proof of Theorem~\ref{Dep=Cyc}] 
Assume the contrary, then there is a subset $X\subset E$, which is dependent and without cycles, i.e. $X$ is a spanning forest.

By Theorem~\ref{map} we know that there is map $g$ from $X$ to the pairs of vertices, which gives a usual forest.
Consider the vector set $$\{a_e:=z_{g_1(e)}-z_{g_2(e)},\ e\in X\},$$ where $z_v=(0,\ldots,0,1,0,\ldots,0)$ is the unit vertex correspinding to the vertex $v$.
Since $g$ gives the usual forest, these vectors are independent. Hence,  generic vectors $\{b_e,\ e\in X\}$ are also linear independent.
We get that the edges are independent, contradiction. 
\end{proof}

By the induced subgraph on vertices $V'\subset V$, we assume 
a hypergraph $(V',E')$, where $E'$ are all edges of $E$, which
have vertices only from $V'$ (i.e. $e\in E'$ if $e\subset V'$).
This definition works well with colorings of hypergraphs,
because if we want to color a hypergraph such a way that there are no
monochromatic edges, then it is the same as splitting vertices into sets with empty induced subgraphs. Also this definition works well
with standard sense of connectivity.

\begin{prop}\label{str-components} Let $V_1$ and $V_2$ be the subsets of vertices such that the induced subgraphs of $H$ on $V_j$
are strongly connected and $V_1\cap V_2\neq \emptyset$. Then
the induced subgraph of $H$ on $V_1\cup V_2$ is also strongly
connected 
\end{prop} 
\begin{proof}
For any vertex $v_i\in V$ we will consider  the corresponding unit vector
$z_{v_i}=(0,\ldots,0,1,0,\ldots,0).$
Fix a vertex $u$, which lie in the intersection $V_1\cap V_2$. 
Let $H_1$ and $H_2$ be induced subgraphs on $V_1$ and $V_2$, resp.

Consider vectors $b_i$ corresponding to $e_i\in E_1$. 
We know that there is a spanning tree of the graph $H_1$, then the dimension of the linear space of such vectors is $|V_1|-1$, furthermore any sum
of coordinates of any vector is zero. Hence, for any $v\in V_1$, $z_v-z_u\in\text{span}\{b_i:\ e_i\in E_1\}.$ Similarly we get
the same for $H_2$ and, hence, we have the same for $H_1\cup H_2$.

We get that for any $v\in V_1\cup V_2$, $z_v-z_u\in\text{span}\{b_i:\ e_i\in E_1\cup E_2\}.$ 
Hence, the hypergraph $H_1\cup H_2$ has $|V_1\cup V_2|-1$ independent edges, then has a spanning tree. 
We have that the induced subgraph of $H$ on vertices $V_1\cup V_2$ is strongly connected, since it
has all edges from $H_1\cup H_2$. 
\end{proof}

\bigskip

\section{\bf Algebras corresponding to spanning trees, Problems}

In this section we discuss analogous algebras counting spanning trees.
Recall the definition of algebras ${\cal B}_G^T$ and ${\cal C}_G^T$ borrowed from~\cite{PSh}. 
\begin{Not}
\label{deftree}
 Take an undirected graph $G$ with $n$ vertices.

\emph{(I)} Let $\Phi_{G}^{T}$ be the commutative algebra over $\bK$ generated by $\{ \phi_e: \ e\in G \}$ with relations $\phi_e^{2}=0$, for any $e\in G$, and $\prod_{e\in H}\phi_e=0$, for any cut $H \subset E(G)$.

Fix a linear order of vertices of $G$. 
For $i=1,\ldots,n$, set 
$$X_i=\sum_{e\in E(G)} c_{i,e} \phi_e,$$
 where $c_{i,e}$ as in Notation~\ref{forest}.
Denote by ${\cal C}_{G}^{T}$ the subalgebra of $\Phi_{G}^{T}$ generated by $X_1,\ldots ,X_n$.

\emph{(II)}   Consider the ideal $J_{G}^{T}$ in the ring $\bK[x_1,\cdots,x_n]$ generated by $$p_{[n]}=x_1+\ldots+x_n$$ and by
$$p_I^{T}=\left(\sum_{i\in I} x_i\right)^{D_I},$$
where $I$ ranges over all nonempty proper subsets of vertices, and $D_I$ is the total number of edges between vertices in $I$ and vertices outside $I$.  Define the algebra   ${\cal B}_{G}^{T}$ as the quotient $\bK[x_1,\dots,x_n]/ J_{G}^{T}.$
\end{Not}

The case of disconnected graphs  is not interesting, because both algebras are trivial. In the paper~\cite{PSh} the following result was proved:

\begin{thm}[cf.~\cite{PSh}]
\label{thmPShtree}

 For any connected graph $G$, algebras  ${\cal B}_{G}^{T}$ and ${\cal C}_{G}^{T}$ are isomorphic;
 their total dimension over $\bK$ is equal to the number of spanning trees in $G$.

Moreover, the dimension of the $k$-th graded component of these algebras  equals
the number of spanning trees of $G$ with external activity $e(G)-v(G)+1-k$.
\end{thm}

\begin{cor}
Given a connected graph $G$, the Hilbert series of the algebra ${\cal C}_G^T$ is given by
$$\HS_{{\cal C}_G^T}(t)=T_G\left(1,\frac{1}{t}\right)\cdot t^{e(G)-v(G)+c(G)},$$ 
where $c(G)$ is the number of connected components of $G$.
\end{cor}

\subsection{\bf Algebras and matroids} 
For graph $G$, we define its bridge-free matroid  as the usual graphical matroid of the graph $G'$ which is obtained from $G$ after removing  all its bridges.

\begin{prop}
\label{treematroid}
 For any pair of connected graphs $G_1$ and $G_2$ with isomorphic bridge-free matroids, their algebras ${\cal B}_{G_1}^T$ and ${\cal B}_{G_2}^T$ are isomorphic.

\begin{proof}
Notice that, if we add an edge $e$ and a vertex $v$ to $G$, such that $v$ is an endpoint of $e$ and another endpoint of $e$ is some vertex of $G$, then algebra ${\cal B}^T$ does not change (this is obvious because ${e}$ is a bridge, and hence, $\phi_e$ is one of the generators of the ideal). This operation doesn't change bridge-free matroid. Therefore it is enough to prove Proposition~\ref{treematroid} only for graphs with the same number of edges.

Assume that $|E(G_1)|=|E(G_2)|$. In this case an isomorphism of bridge-free matroids is equivalent to an isomorphism of matroids of graphs $G_1$ and $G_2$.

In fact, in Lemma~\ref{forestlem} we construct orientations $\overline{G}_1$ and $\overline{G}_2$ of graphs $G_1$ and $G_2$ on the same set of edges $E(G_1)=E(G_2)$ (it was constructed for graphs differ in one Whitney's deformation, so we can  extend it to a sequence of deformations), such that they give the same graphical matroid on edges and with these orientations the algebras ${\cal C}_{\overline{G}_1}^F$ and ${\cal C}_{\overline{G}_2}^F$  coincide as subalgebras of~$\Phi_{G_i}^F$ ($\Phi_{G_1}^F$ and $\Phi_{G_2}^F$ are the same, because graphs have common set of edges).

Let $I$ be the ideal generated by the products of edges from the cuts of $G_1$ in $\Phi_{G_1}^F$. Because the variables on edges in $G_1$ and $G_2$ are the same and $C$ is a cut in $G_1$ if and only if $C$ is a cut in $G_2$, then $I$ is also the ideal generated by the cuts of~$G_2$.

Thus we have $\Phi_{G_1}^T=\Phi_{G_1}^F/I,$ hence, 
$${\cal C}_{\overline{G}_1}^T={\cal C}_{\overline{G}_1}^F/I,$$
similarly 
$${\cal C}_{\overline{G}_2}^T={\cal C}_{\overline{G}_2}^F/I.$$
It means that the algebras ${\cal C}_{G_1}^T$ and ${\cal C}_{G_2}^T$ are also the same in orientations $\overline{G}_1$ and $\overline{G}_2$. 
\end{proof}
\end{prop}

We formulate the following converse conjecture.

\begin{conjecture}
\label{conj}
 Algebras ${\cal B}_{G_1}^T$ and ${\cal B}_{G_2}^T$ for the connected graphs $G_1$ and $G_2$ are isomorphic if and only if 
 their bridge-free matroids are isomorphic.
\end{conjecture}

\subsection{\bf $t$-labelled trees}

It is possible to introduce similar algebras which enumerate $t$-labelled trees, but it is not very exciting. Let ${\cal B}_G^{T_t}$ be an algebra in which we change the generators $(\sum_{i\in I} x_i)^{D_I}$ of the ideal by $(\sum_{i\in I} x_i)^{tD_I}.$ The definition of ${\cal C}_G^{T_t}$ will change in a more complicated way.  

However, there is no result about a reconstruction of the Tutte polynomial from the Hilbert series, because all trees have the same number of edges and then
$\HS_{{\cal B}_G^{T_t}}(x)=(1+x)^{n-1}\HS_{{\cal B}_G^{T}}(x^t).$ In other words, the Hilbert series of ${\cal B}_G^{T_t}$ and of ${\cal B}_G^{T}$ contain the same information about the graph.

\subsection{\bf Algebras for hypergraphs.}
The main problem is to construct a family $\widehat{\cal C}^T_H$ of algebras, which count spanning trees of $H$.

By paper~\cite{PSS}, for  a hypergraph $H$ and a set of parameters $C$, we can present $\cal{C}_{H(C)}^F$ as a quotient algebra, i.e, as $\cal{B}_{H(C)}^F$. We can consider the algebra  $\cal{B}_{H(C)}^T$, which is obtained from $\cal{B}_{H(C)}^F$ by changing   the powers of the generators of the ideal (writing always one less).  By paper~\cite{ArP} algebra $\cal{B}_{H(C)}^T$ should count spanning trees of~$H$. However at this moment, we can not present $\Phi_H^T$ such that its generic subalgebra $\cal{C}_{H(C)}^T$ counts spanning trees of~$H$.

\smallskip
Probably, we need to add to $\Phi_H^F$ relations corresponding to cuts, where a cut is a subset of edges such that without it $H$ has no spanning trees. However, we need to prove it and if we want to do something similar to the proof of Theorem~\ref{thmPShtree}, then we need to define $H$-parking functions for a hypergraph.


\begin{thebibliography}{8}

\bibitem{Abd} A.\,Abdesselam, {\it The Grassmann-Berezin calculus and theorems of the matrix-tree type}, 
Advances in Applied Mathematics 33(1) (2004), pp. 51-70.

\bibitem{ArP} F.\,Ardila, A.\,Postnikov, {\it Combinatorics and geometry of power ideals,} Trans. Amer. Math. Soc., 362(8) (2010), pp. 4357-4384.

\bibitem{Ber} C.\,Berge, {\it Hypergraphs}, 
North-Holland Mathematical Library 45, 1989.

\bibitem{Berg} A.\,Berget, {\it Products of linear forms and Tutte polynomials,} European J. Combin. 31 no. 7  (2010), pp. 1924-1935.
\bibitem{BrOx} T.\,Brylawski, J.\,Oxley, {\it The Tutte polynomial and its applications,} chapter in Matroid Applications, Encyclopedia of Mathematics and its Applications 40 (1992).

\bibitem{Bur} Y.\,Burman, A.\,Ploskonosov, A.\,Trofimova, {\it Matrix-tree theorems and discrete path integration,} Linear Algebra and its Applications, 466 (2014), pp. 64-82.

\bibitem{B&CO} A.\,Brandstadt, F.\,Dragan, V.\,Chepoi, V.\,Voloshin, {\it Dually chordal graphs}, 
SIAM Journal on  Discrete Mathematics 11(3) (1998), pp. 437-455.

\bibitem{ChK} S.\,Chaiken, D.\,J.\,Kleitman, {\it Matrix Tree Theorems,}  
Journal of Combinatorial Theory, Series A 24 (1978), pp. 377-381.

\bibitem{ChP} D.\,Chebikin, P.\,Pylyavskyy, {\it A family of bijections between $G$-parking functions and spanning trees,} Journal of Combinatorial Theory, Series A 110 (2005), pp. 31-41.
\bibitem{Dhar} D.\,Dhar, {\it Self-organised critical state of the sandpile automaton models,} Physical Review
Letters 64 no. 14 (1990), 1613-1616.

\bibitem{FK} S.\,Fomin, A.\,N.\,Kirillov, {\it Quadratic algebras, Dunkl elements, and Schubert calculus,} Advances in geometry, volume 172 of Progr. Math. (1999), pp. 147-182. 
  
\bibitem{Hall} P.\,Hall, {\it On Representatives of Subsets}, 
J.London Math. Soc. 10(1) (1935), pp. 26-30.





\bibitem{HR} O.\,Holtz, A.\,Ron, {\it Zonotopal algebra}, Advances in Mathematics, 227 (2011), pp. 847-894.
  
  
\bibitem{Hua} B.\,Huang, {\it Monomization of Power Ideals and Generalized Parking
Functions,} \url{https://math.mit.edu/research/highschool/primes/materials/2014/Huang.pdf}.

\bibitem{Kal} T.\,K{\'a}lm{\'a}n, {\it A version of Tutte's polynomial for hypergraphs}, 
Advances in Mathematics 244 (2013), pp. 823-873.

\bibitem{Kir} G.\,Kirchhoff, {\it Uber die Aufl{\"o}sung der Gleichungen, auf welche man bei der untersuchung der linearen verteilung galvanischer Str{\"o}me gef{\"u}hrt wird,} Ann. Phys. Chem. 72 (1847), pp. 497-508.

\bibitem{Lenz} M.\,Lenz, {\it Hierarchical Zonotopal Power Ideals,} 
European Journal of Combinatorics 33 (2012), pp. 1120-1141.

\bibitem{RL} R.\,I.\,Liu, {\it On the commutative quotient of Fomin-Kirillov algebras,} European Journal of Combinatorics
Volume 54 (2016), pp. 65-75.

\bibitem{NS} G.\,Nenashev, B.\,Shapiro, {\it "K-theoretic" analog of Postnikov-Shapiro algebra distinguishes graphs,} Journal of Combinatorial Theory, Series A, 148 (2017), pp. 316-332.

\bibitem{OT} P.\,Orlik, H.\,Terao, {\it Commutative algebras for arrangements,} Nagoya Math. J. 134 (1994), pp. 65-73.

\bibitem{Oxl} J.\,G.\,Oxley, {\it Matroid theory,} Oxford University Press, New York (1992).

\bibitem{PSh} A.\,Postnikov, B.\,Shapiro, {\it Trees, parking functions, syzygies, and deformations of monomial ideals,}
 Trans. Amer. Math. Soc. 356 (2004), pp. 3109-3142.
 
\bibitem{PSS} A.\,Postnikov, B.\,Shapiro, M.\,Shapiro, {\it Algebra of curvature forms on homogeneous manifolds}, AMS Trans. Ser 2, vol. 194 (1999), pp. 227-235.







\bibitem{ShSh} B.\,Shapiro, M.\,Shapiro, {\it On ring generated by Chern 2-forms on $SL_n/B$}, C. R. Acad. Sci. Paris S{\'e}r. I Math. 326(1) (1998), pp. 75-80.
\bibitem{S} A.\,Sokal, {\it The multivariate Tutte polynomial (alias Potts model) for graphs and matroids}, In
Surveys in Combinatorics (2005), pp. 173-226.

\bibitem{Tut} W.\,T.\,Tutte, {\it Graph Theory,} Cambridge University Press (2001).

\bibitem{Wit} H.\,Whitney, {\it 2-isomorphic graphs,} American Journal of Mathematics 55 (1933),  pp. 245-254.

\end{thebibliography}
\end{document}